\def\lmjrunauthor{D. Stankov}
\def\lmjruntitle{The reciprocal algebraic integers having small house}

\def\lmjmsc{[2010] 11C08 \sep 11R06 \sep 11Y40}
\def\lmjkeywords{algebraic integer \sep the house of algebraic integer \sep maximal modulus \sep reciprocal polynomial \sep primitive polynomial \sep Schinzel-Zassenhaus conjecture \sep
Mahler measure \sep method of least squares \sep cyclotomic polynomials}

\documentclass{lmj}

\usepackage{lineno,hyperref}
\usepackage{color}
\usepackage[T1]{fontenc}
\definecolor{darkblue}{rgb}{0,0,0.8}
\hypersetup{colorlinks,
            linkcolor=darkblue,
            anchorcolor=darkblue,
            citecolor=darkblue}
\let\orgautoref\autoref
\renewcommand{\autoref}[1]
        {\def\equationautorefname~##1\null{~(##1)\null}%
         \orgautoref{#1}}

\hypersetup{colorlinks=true}

\modulolinenumbers[5]

\def\house#1{{%
    \setbox0=\hbox{$#1$}
    \vrule height \dimexpr\ht0+1.4pt width .4pt depth \dp0\relax
    \vrule height \dimexpr\ht0+1.4pt width \dimexpr\wd0+2pt depth \dimexpr-\ht0-1pt\relax
    \llap{$#1$\kern1pt}
    \vrule height \dimexpr\ht0+1.4pt width .4pt depth \dp0\relax
}}

\def\roof#1{{%
    \setbox0=\hbox{$#1$}
    \vrule height \dimexpr\ht0+1.4pt width .8pt depth \dp0\relax
    \vrule height \dimexpr\ht0+1.8pt width \dimexpr\wd0+2pt depth \dimexpr-\ht0-1pt\relax
    \llap{$#1$\kern1pt}
    \vrule height \dimexpr\ht0+1.4pt width .8pt depth \dp0\relax
}}

\usepackage{amssymb}
\usepackage{amsmath}
\usepackage{longtable}

\begin{document}
\bibliographystyle{plainlmj}

\begin{topmatter}



\title{The reciprocal algebraic integers having small house}


\author{Dragan Stankov}
\institution{Katedra Matematike RGF-a 11000 Beograd, \DJ u\v{s}ina 7, Serbia}
\email{dstankov@rgf.bg.ac.rs}

\begin{abstract}
Let $\alpha$ be an algebraic integer of degree $d$, which is reciprocal. The house of $\alpha$ is the largest modulus of its conjugates. We proved that $d$-th power of the house of reciprocal $\alpha$ has a limit point. We presented a property of antireciprocal hexanomials. We compute the minimum of the houses of all reciprocal algebraic integers of degree $d$ having the minimal polynomial which is a factor of a $D$-th degree
reciprocal or antireciprocal polynomial with at most eight monomials, say $\mathrm{mr}(d)$, for $d$ at most 180, $D\le 1.5d$ and $D\le 210$. We show that it is not necessary to take into account unprimitive polynomials.
The computations suggest several conjectures.

\end{abstract}

\Keywords




\end{topmatter}


\section{Introduction}
\label{intro}
Let $\alpha$ be an algebraic integer of degree $d$, with conjugates
$\alpha=\alpha_1, \alpha_2,\ldots,\alpha_d$ and minimal
polynomial $P$. The house of $\alpha$ (and of $P$) is defined by:
\[\house{\alpha} = \max\limits_{1\leq i\leq d}|\alpha_i|.\]
The Mahler measure of $\alpha$ is $M(\alpha) = \prod_{i=1}^{d}
\max(1, |\alpha_i|)$.
Clearly,
$\house{\alpha} > 1$, and a theorem of Kronecker 
tells us that $\house{\alpha} = 1$ if and only if $\alpha$ is a root
of unity. In 1965, Schinzel and Zassenhaus \cite{schinzel1965} have made the following conjecture:
\begin{conjecture}[SZ]
There is a constant $c > 0$ such that if $\alpha$ is not a root of unity, then $\house{\alpha}\ge 1 + c/d$.
\end{conjecture}

If a polynomial has only eight non-zero coefficients then it is called an octanomial. Similarly, if the number
of non-zero coefficients is five, six and seven, such polynomial is called pentanomial, hexanomial and heptanomial respectively.
A polynomial $P(x)$ of degree $d$ is antireciprocal if it satisfies $P(x)=-x^dP(1/x)$.
The height of a polynomial is defined to be the maximum of the moduli of its coefficients.
Let $\mathrm{m}(d)$ denote the minimum of $\house{\alpha}$ over $\alpha$ of degree $d$ which are not roots of unity.
Let an $\alpha$ attaining $\mathrm{m}(d)$ be called extremal. We say that $\alpha$ is reciprocal if $\alpha^{-1}$ is a conjugate
of $\alpha$, i.e. $X^dP(1/X) = P(X)$. Let $\mathrm{mr}(d)$ denote the minimum of $\house{\alpha}$ over reciprocal $\alpha$ of degree $d$ which are not roots of
unity. 
Let an $\alpha$ attaining $\mathrm{mr}(d)$ be called extremal reciprocal.
In 1985, D. Boyd \cite{10.2307/2008062} conjectured, using a result of C.J. Smyth \cite{Smy1}, that $c$ should be equal to $3/2 \log \theta$ where $\theta = 1.324717 \ldots$ is the smallest Pisot number, the real root of the polynomial $x^3 - x - 1$. Intending to verify his conjecture that extremal $\alpha$ are always nonreciprocal, Boyd has computed the smallest houses for reciprocal polynomials of even degrees $\le 16$. Wu and Zhang \cite{WU2017170} continued the Boyd's computation with even degrees $\le 42$. They showed in their Table 5 that the minimal polynomial of extremal reciprocal algebraic integer can be written as a factor of a reciprocal polynomial with at most eight monomials. The same fact is valid for many polynomials having Mahler measure less than 1.3 and has been used for creation of the Mossinghoff's list of such polynomials \cite{Mossinghoffwebsite}. We used here this idea to search for extremal reciprocals of degree $d$ having the minimal polynomial which is a factor of a $D$-th degree reciprocal or antireciprocal hexanomial or octanomial, where $d$ is at most 180 and $D$ is at most 210.

\section{Theorems and proofs}

A polynomial $P(x)$ is primitive if it cannot be expressed as a polynomial in $x^k$, for some
$k \ge 2$. Clearly, if $p$ is an odd prime number then any reciprocal polynomial of degree $2p$ with more than three monomials has to be primitive.
It is easy to verify that
\begin{equation}\label{Unprim}
\house{P(x^k)} = \sqrt[k]{\house{P(x)}}.
\end{equation}
Let $\mathrm{mrp}(d)$ denote the minimum of $\house{\alpha}$ over reciprocal algebraic integer $\alpha$ of degree $d$ which are not roots of unity and which have a primitive minimal polynomial.
Let $\mathrm{mrp}(d)$ is attained for $\alpha_d$ with minimal reciprocal primitive polynomial $R_d(x)$. Let $\alpha_d$ be called extremal reciprocal primitive. Clearly
\begin{equation}\label{mpleqmrp}
\mathrm{mr}(d) \leq \mathrm{mrp}(d),
\end{equation}
and the equality is strict if and only if the $\alpha$ attaining $\mathrm{mr}(d)$ is not a root of a primitive polynomial.
\begin{lemma}\label{sec:PolProdRec}
Let $k_i$, $k_j$ be integers and $d_i$, $d_j$ be even integers such that $k_id_i=k_jd_j=d$. If $\mathrm{mrp}^{d_i}(d_i)<\mathrm{mrp}^{d_j}(d_j)$ then the house of $R_{d_i}(x^{k_i})$ is less than the house of $R_{d_j}(x^{k_j})$.
\end{lemma}
\begin{proof}
. Raising both sides of $\mathrm{mrp}^{d_i}(d_i)<\mathrm{mrp}^{d_j}(d_j)$ to the power $1/d$ we obtain $\mathrm{mrp}^{1/k_i}(d_i)<\mathrm{mrp}^{1/k_j}(d_j)$. It remains to recall that the house of $R_{d_i}(x^{k_i})$ is equal to $\mathrm{mrp}^{1/k_i}(d_i)$ and the house of $R_{d_j}(x^{k_j})$ is equal to $\mathrm{mrp}^{1/k_j}(d_j)$.\qed
\end{proof}

\begin{corollary}\label{sec:compositeRec}
Let $d$ be an even natural number and let $d_0,d_1,\ldots,d_m$ be all even natural divisors of $d$. Let $\mathrm{mrp}(d_j)$ is attained for a reciprocal $\alpha_{d_j}$ with minimal polynomial $R_{d_j}(x)$ where, $R_{d_j}(x)$ is a primitive reciprocal polynomial, $k_j=d/d_j$, $j=0,1,\ldots,m$.
If $$\mathrm{mrp}^{d_0}(d_0)<\mathrm{mrp}^{d_1}(d_1)<\cdots<\mathrm{mrp}^{d_m}(d_m)$$ then
$\mathrm{mr}(d)=\mathrm{mrp}^{1/k_0}(d_0)$.
\end{corollary}
\begin{proof}
. Straightforwardly from Lemma \ref{sec:PolProdRec} it follows that polynomial $R_{d_0}(x^{k_0})$ has the house which is less than the house of any other polynomial of degree $d$.
Therefore 
\begin{eqnarray*}
\mathrm{mr}(d)&= & R_{d_0}(x^{k_0})\\
&=&\mathrm{mrp}^{1/k_0}(d_0).
\end{eqnarray*}
\qed
\end{proof}

The smallest limit point of the Mahler measure is believed to be $1.255433 \ldots$ \cite{peterborwein2002topics} which appears from the sequence of reciprocal heptanomials $(Q_n(x))_{n\ge 1}$,
$$Q_{n}(x)=x^{2n}+x^{2n-1}+x^{n+1}+x^n+x^{n-1}+x+1.$$
What is the smallest limit point of $\house{P_d(x)}^d$ where $P_d(x)$ is a reciprocal polynomial of degree $d$ arises as an interesting problem.
\begin{lemma}\label{sec:mddBoundedRec}
The sequence $(\mathrm{mr}^d(d))_{d\ge 1}$ is bounded. If $\alpha=3/2+\sqrt{5}/2=2.618\ldots$ then $\alpha^2=6.854\ldots$ is an upper bound.
\end{lemma}
\begin{proof}
. If $\mathrm{mr}(d)$ is attained for $\alpha_d$ then $$\mathrm{mr}(d)=\house{\alpha_d}\le \house{x^d+3x^{d/2}+1}= \sqrt[d/2]{2.618\ldots}.$$ The claim follows straightforwardly if we raise both sides of the inequality to the power $d$.
\qed
\end{proof}


It is much more difficult to prove that the sequence $(\mathrm{mrp}^d(d))_{d\ge 1}$ is bounded.
Let us introduce the following sequence of polynomials $(P_{2n}(x))_{n\ge 2}$ $$P_{2n}(x)=x^{2n}-x^{n+1}-x^n-x^{n-1}+1.$$ They are obviously reciprocal and primitive.
\begin{theorem}\label{sec:RealRoot}
There is a unique real root $\alpha_n \in (\sqrt[n]{2},\sqrt[n]{3})$ of $P_{2n}(x)$ such that the house of $P_{2n}(x)$ is equal to $\alpha_n$.
\end{theorem}
\begin{proof}
. There is a real root $\alpha_n \in (\sqrt[n]{2},\sqrt[n]{3})$ of $P_{2n}(x)$ because $P(\sqrt[n]{2})=2^2-2^{(n+1)/n}-2-2^{(n-1)/n}+1<1-2^{(n-1)/n}<0$ and
$P(\sqrt[n]{3})=3^2-3^{(n+1)/n}-3-3^{(n-1)/n}+1>2^2-2^{(n+1)/n}>0$.
If we divide the equation $P_{2n}(x)=0$ with $x^n$ we get
$$x^{n}-x-1-x^{-1}+x^{-n}=0.$$
Afterwards we use the trigonometric form of complex $x$, $x=r(\cos(\varphi)+i\sin(\varphi))$ and get
$$r^{n}(\cos n\varphi+i\sin n\varphi)-r(\cos \varphi+i\sin \varphi)-1-r^{-1}(\cos(-\varphi)+i\sin(-\varphi))+r^{-n}(\cos(-n\varphi)+i\sin(-n\varphi))=0$$
Separating real and imaginary part of this equation we conclude that every root has to satisfy the following system:
\begin{align}\label{intro:RealImag1}
\left(r^{n}+\frac{1}{r^n}\right)\cos(n\varphi)-\left(r+\frac{1}{r}\right)\cos(\varphi)-1&=0,\\
\left(r^{n}-\frac{1}{r^n}\right)\sin(n\varphi)-\left(r-\frac{1}{r}\right)\sin(\varphi)&=0.\label{intro:RealImag2}
\end{align}
If a root is unimodal i.e. $r=1$ then \eqref{intro:RealImag2} is clearly satisfied. Equation \eqref{intro:RealImag1} then gives
$2\cos(n\varphi)-2\cos(\varphi)-1=0$,
and afterwards
\begin{equation}\label{intro:Eqcos}
  \cos(n\varphi)=\cos(\varphi)+0.5.
\end{equation}
Since the left side of \eqref{intro:Eqcos} $\cos(n\varphi)\in [-1,1]$ it follows that on the right side $\cos\varphi \in [-1,0.5])$ $\Leftrightarrow$ $\varphi \in (\pi/3,5\pi/3)$. If $\varphi \in (-\pi/3,\pi/3)$ then there are no unimodal roots because the left side of \eqref{intro:Eqcos} can not be equal to the right side. If $r>1$ we can solve \eqref{intro:RealImag1} for $\cos(n\varphi)$ and \eqref{intro:RealImag2} for $\sin(n\varphi)$ in terms of $\varphi$ and $r$. Finally we get an implicit equation of a plane curve $r=r(\varphi)$ which contains all roots which are not unimodal when $\varphi \in (-\pi/3,\pi/3)$:
\begin{equation}\label{intro:Impl}
\left[\frac{\left(r-\frac{1}{r}\right)\sin(\varphi)}{r^{n}-\frac{1}{r^n}}\right]^2+\left[\frac{\left(r+\frac{1}{r}\right)\cos(\varphi)+1}{r^{n}+\frac{1}{r^n}}\right]^2=1.
\end{equation}

\begin{figure}[t]
  \includegraphics[width=0.7\textwidth]{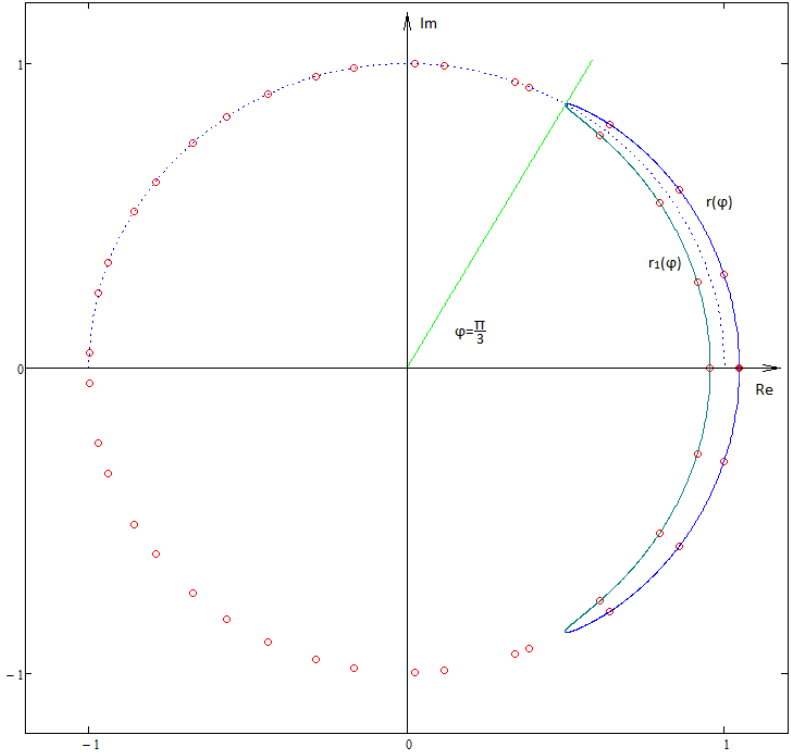}
\caption{Roots of the reciprocal polynomial $P_{42}(x)=x^{42}-x^{22}-x^{21}-x^{20}+1$ are represented with $\circ$. If modulus of a root is equal to one then its argument $\in (\pi/3,5\pi/3)$, else the root is lying on the graph of $r(\varphi)$, the solution of the polar 
equation \eqref{intro:Impl}. The root having maximum modulus is the real root denoted with $\bullet$. }
\label{fig:1}       
\end{figure}

It is obvious that if $r(\varphi)$ satisfies \eqref{intro:Impl} then $r_1(\varphi)=1/r(\varphi)$ also satisfies \eqref{intro:Impl}. We can calculate $\theta$ such that $r(\theta)=1$. If we use that
$r^n-\frac{1}{r^n}=\left(r-\frac{1}{r}\right)(r^{n-1}+r^{n-3}+\cdots+r^{-(n-1)}$ and replace $r=1$ in \eqref{intro:Impl} we obtain a quadratic equation
$$\frac{1-\cos^2(\theta)}{n^2}+\frac{(2\cos(\theta)+1)^2}{4}=1.$$
and solve it for $\cos(\theta)$.
The solution
$\cos(\theta)=\frac{\sqrt{4n^4-7n^2+4}-n^2}{2n^2-2}\in(0,\frac{1}{2})$ and tends to $0.5$ when $n$ tends to infinity. It follows that $\pi/3<\theta<\pi/2$ and $r(\varphi)>1$ on $(-\theta,\theta)\supset (-\pi/3,\pi/3)$ (see Figure \ref{fig:1}).
The claim will be proved if we show that $r(\varphi)$ increase on $(-\pi/3,0)$ and decrease on $(0,\pi/3)$.
If we denote the left side of \eqref{intro:Impl} by $F(r(\varphi),\varphi)$ then the implicit equation becomes $F(r(\varphi),\varphi)=1$. It is well known that the first derivative of $r(\varphi)$ is $$r'(\varphi)=-\frac{\frac{\partial F}{\partial \varphi}}{\frac{\partial F}{\partial r}}.$$
Since
\begin{equation}\label{intro:Impl1}
\frac{\partial F}{\partial \varphi}=\frac{2\sin(\varphi)\cos(\varphi)(r-\frac{1}{r})}{r^{n}-\frac{1}{r^n}}-\frac{2\sin (\varphi) (r+\frac{1}{r})( (r+\frac{1}{r})\cos(\varphi)+1)}{(r^{n}+\frac{1}{r^n})^2}
\end{equation}
can be simplified
$$-\frac{ r^{4n}-2r^{2n}+1+r^{4n+2}-2r^{2n+2}+r^2+4r^{4n+1}\cos\varphi-4r^{2n+3}\cos\varphi-4r^{2n-1}\cos\varphi+4r\cos\varphi }{(2r^{2n-1}\sin\varphi)^{-1}(r^{2n}-1)^2(r^{2n}+1)^2}$$
and factored into
$$-\frac{2r^{2n-1}\sin\varphi \left((r^{2n}-1)^2(r^2+1)+4r\cos\varphi(r^{2n+2}-1)(r^{2n-2}-1) \right) }{(r^{2n}-1)^2(r^{2n}+1)^2}$$
it follows that
\begin{equation}\label{FoverfiNeg}
\varphi \in (-\pi/3,0) \Rightarrow \frac{\partial F}{\partial \varphi}>0,\;\;\; \varphi \in (0,\pi/3,0) \Rightarrow \frac{\partial F}{\partial \varphi}<0.
\end{equation}
so that the claim will be proved if we show that $-\frac{\partial F}{\partial r}>0$ on $(-\pi/3,\pi/3)$.
We derive that
$$\frac{\partial F}{\partial r}=\frac{2\sin^2 \varphi \left(nr^{n-1}+\frac{n}{r^{n+1}}\right)\left(r-\frac{1}{r}\right)^2}{\left(\frac{1}{r^n }-r^n\right)^3}-\frac{2\left[\cos \varphi \left(r+\frac{1}{r}\right)+1\right]^2\left(nr^{n-1}-\frac{n}{r^{n+1}}\right) }{\left(\frac{1}{r^n }+r^n\right)^3}+$$
$$+\frac{2\sin^2 \varphi \left(r-\frac{1}{r}\right)\left(\frac{1}{r^2}+1\right)}{\left(\frac{1}{r^n }-r^n\right)^2}-\frac{2 \cos \varphi \left[\cos \varphi \left(r+\frac{1}{r}\right)+1\right]\left(\frac{1}{r^{2}}-1\right) }{\left(\frac{1}{r^n }+r^n\right)^2}.$$
can be expanded
$$\frac{\partial F}{\partial r}=-\frac{2r^{2n-3}\left(4r^2A(r)\cos^2\varphi+B(r)\cos\varphi+C(r)\right)}{\left(r^{4n}-1\right)^3}$$
where 
\begin{align*}
B(r)= &(2n-1)r^{8n+3}+(2n+1)r^{8n+1}-(8n-2)r^{6n+3}-(8n+2)r^{6n+1}+12nr^{4n+3}+\\
& +12nr^{4n+1}-(8n+2)r^{2n+3}-(8n-2)r^{2n+1}+(2n+1)r^3+(2n-1)r\\
= &r(r^{2n}-1)^3((2n-1)(r^{2n+2}-1)+(2n+1)r^2(r^{2n-2}-1)) 
\end{align*}
so that we can conclude that $B(r)$ is greater than $0$ when $r>1$.
We intend to prove that also
$$A(r)=nr^{8n}-(2n-1)r^{6n+2}-(2n+1)r^{6n-2}+6nr^{4n}-(2n+1)r^{2n+2}-(2n-1)r^{2n-2}+n$$ and
$$C(r)=(n-1)r^{8n+4}-nr^{8n+2}+(n+1)r^{8n}+(4n-2)r^{6n+4}-12nr^{6n+2}+(4n+2)r^{6n}+$$
$$+6nr^{4n+4}-6nr^{4n+2}+6nr^{4n}+(4n+2)r^{2n+4}-12nr^{2n+2}+(4n-2)r^{2n}+(n+1)r^{4}-nr^{2}+n-1 $$
are greater than $0$ when $r>1$. If we substitute $r^2$ in $A(r)$ with $t+1$ then we have to prove that
$$D_{4n}(t)=n(t+1)^{4n}-(2n-1)(t+1)^{3n+1}-(2n+1)(t+1)^{3n-1}+$$
$$+6n(t+1)^{2n}-(2n+1)(t+1)^{n+1}-(2n-1)(t+1)^{n-1}+n>0$$ when $t>0$.
If we develop $D_{4n}(t)$ and denote its coefficients by $a_k$ i.e. $D_{4n}(t)=a_{4n}t^{4n}+a_{4n-1}t^{4n-1}+\cdots+a_1t+a_0$ then it is sufficiently to prove that 
$a_k=\underline{n\binom{4n}{k}-(2n-1)\binom{3n+1}{k}-(2n+1)\binom{3n-1}{k}}+(6n)\binom{2n}{k}-(2n+1)\binom{n+1}{k}-(2n-1)\binom{n-1}{k}+n\binom{0}{k}$
are $\geq 0$ where we use the convention that
$$\binom{n}{k}=\frac{n(n-1)\cdots(n-k+1)}{k!},\;\;0<k\le n,\;\; \binom{n}{0}=1,\;\; \binom{0}{0}=1,\;\; \binom{n}{k}=0,\;\;k>n.$$
We can verify that $a_k=0,\;\;k=0,1,2,3$ and that $a_4=n(n-1)(n+1)(3n^2-1)/3>0$. Let us denote underlined part of $a_k$ with $b_k$ and the remainder with $c_k$. For $k\geq 5$ it is sufficient to prove that $b_k>0$ and $c_k>0$. If $n\geq k>3n+1$ then, using the convention, $b_k=n\binom{4n}{k}$ so that $b_k>0$. If $k\leq 3n-1$ then we prove that $b_k>0$, using the principle of mathematical induction. For $k=5$ and $n\geq 2$ we have
\begin{align*}
b_5 &=n\binom{4n}{5}-(2n-1)\binom{3n+1}{5}-(2n+1)\binom{3n-1}{5}\\
&=(n-1)(26n^5+366n^4-539n^3-39n^2+240n-60)/60\\
&>0.
\end{align*}
Assume that the inequality is true for $k=m$, $m<3n-1$ and $m\ge 5$ and let us prove that it is true for $k=m+1$. For $k=m$ we have $0<b_m$ i.e.
$$0<n\binom{4n}{m}-(2n-1)\binom{3n+1}{m}-(2n+1)\binom{3n-1}{m}.$$
If we multiply both sides of this inequality by $(3n+1-m)/(m+1)>0$ then we have
\begin{align*}
0&<n\binom{4n}{m}\frac{3n+1-m}{m+1}-(2n-1)\binom{3n+1}{m}\frac{3n+1-m}{m+1}-(2n+1)\binom{3n-1}{m}\frac{3n+1-m}{m+1}\\
&<n\binom{4n}{m}\frac{4n-m}{m+1}-(2n-1)\binom{3n+1}{m}\frac{3n+1-m}{m+1}-(2n+1)\binom{3n-1}{m}\frac{3n-1-m}{m+1}\\
&=n\binom{4n}{m+1}-(2n-1)\binom{3n+1}{m+1}-(2n+1)\binom{3n-1}{m+1}\\
&=b_{m+1},
\end{align*}
which coincides with the rewritten inequality for $k = m + 1$.
It remains to be proved that $b_{3n}>0$ and $b_{3n+1}>0$. We have just proved that $$0<b_{3n-1}=n\binom{4n}{3n-1}-(2n-1)\binom{3n+1}{3n-1}-(2n+1)\binom{3n-1}{3n-1}$$ so that it follows
$$0<n\binom{4n}{3n-1}-(2n-1)\binom{3n+1}{3n-1}.$$
If we multiply both sides of this inequality by $2/(3n)>0$ then we have
\begin{align*}
0&<n\binom{4n}{3n-1}\frac{2}{3n}-(2n-1)\binom{3n+1}{3n-1}\frac{2}{3n}\\
&<n\binom{4n}{3n-1}\frac{n+1}{3n}-(2n-1)\binom{3n+1}{3n-1}\frac{2}{3n}\\
&=n\binom{4n}{3n}-(2n-1)\binom{3n+1}{3n}=b_{3n}.
\end{align*}
Similarly, if we multiply both sides of the inequality $0<b_{3n}$  by $1/(3n+1)>0$ then we have
\begin{align*}
0&<n\binom{4n}{3n}\frac{1}{3n+1}-(2n-1)\binom{3n+1}{3n}\frac{1}{3n+1}\\
&<n\binom{4n}{3n}\frac{n}{3n+1}-(2n-1)\binom{3n+1}{3n}\frac{1}{3n+1}\\
&=b_{3n+1}.
\end{align*}
Finally it remains to be proved that $c_k>0$. Since
$$c_k=(3n)\binom{2n}{k}-(2n+1)\binom{n+1}{k}+(3n)\binom{2n}{k}-(2n-1)\binom{n-1}{k}+n\binom{0}{k},$$ and it is obviously that
$$(3n)\binom{2n}{k}>(2n+1)\binom{n+1}{k},\;\;\;(3n)\binom{2n}{k}>(2n-1)\binom{n-1}{k}$$ the claim follows.

If we substitute $r^2$ in $C(r)$ with $t+1$ then we have to prove that
$$E_{4n}(t)=(n-1)(t+1)^{4n+2}-n(t+1)^{4n+1}+(n+1)(t+1)^{4n}+(4n-2)(t+1)^{3n+2}-12n(t+1)^{3n+1}+$$
$$+(4n+2)(t+1)^{3n}+6n(t+1)^{2n+2}-6n(t+1)^{2n+1}+6n(t+1)^{2n}+(4n+2)(t+1)^{n+2}-12n(t+1)^{n+1}+$$
$$+(4n-2)(t+1)^{n}+(n+1)(t+1)^{2}-n(t+1)+n-1>0 $$  when $t>0$.
If we develop $E_{4n}(t)$ and denote its coefficients by $d_k$ i.e. $E_{4n}(t)=d_{4n}t^{4n}+d_{4n-1}t^{4n-1}+\cdots+d_1t+d_0$ then it is sufficiently to prove that $d_k=e_k+f_k\geq 0$ where
$$e_{k}=(n-1)\binom{4n+2}{k}-n\binom{4n+1}{k}+(n+1)\binom{4n}{k}+(4n-2)\binom{3n+2}{k}-12n\binom{3n+1}{k}+(4n+2)\binom{3n}{k}$$
$$f_k=6n\binom{2n+2}{k}-6n\binom{2n+1}{k}+6n\binom{2n}{k}+(4n+2)\binom{n+2}{k}-12n\binom{n+1}{k}+$$
$$+(4n-2)\binom{n}{k}+(n+1)\binom{2}{k}-n\binom{1}{k}+(n-1)\binom{0}{k}\geq 0 $$
We can verify that $d_k=0,\;\;k=0,1,2,3$ and that $d_4=n(2n-1)(n+1)(5n^3+n^2+12n+12)/12>0$.
Let us prove that $e_{k}>0$, $3n \geq k\geq 5$. Since
\begin{align}\label{intro:ek}
e_k & =G_1[(n-1)(4n+2)(4n+1)-n(4n+1)(4n-k+2)+(n+1)(4n-k+2)(4n-k+1)]+\\
& + G_2[(4n-2)(3n+2)(3n+1)-12n(3n+1)(3n-k+2)+(4n+2)(3n-k+2)(3n-k+1)]
\end{align}
where $G_1=\frac{4n(4n-1)\cdots (4n-k+3)}{k!}$ and
$G_2=\frac{3n(3n-1)\cdots (3n-k+3)}{k!}$
so that if we divide both sides of the inequality $e_k>0$ with $G_1$ we get the equivalent inequality
$$(n-1)(4n+2)(4n+1)-n(4n+1)(4n-k+2)+(n+1)(4n-k+2)(4n-k+1)+$$
$$+M[(4n-2)(3n+2)(3n+1)-12n(3n+1)(3n-k+2)+(4n+2)(3n-k+2)(3n-k+1)]>0$$ where $M=\frac{3n(3n-1)\cdots (3n-k+3)}{4n(4n-1)\cdots (4n-k+3)}$.
The left side of the inequality is quadratic in $k$ thus we reorder it and then we have
$$(4Mn+n+2M+1)k^2+(12Mn^2-4n^2-12Mn-10n-6M-3)k+$$
$$-36Mn^3+16n^3-36Mn^2+12n^2-8Mn+2n>0$$
so that it is fulfilled if its discriminant $\Delta_1<0$ where
$$\Delta_1=(720n^4+576n^3+416n^2+208n+36)M^2+(-208n^4-176n^3+264n^2+208n+36)M+$$
$$-48n^4-32n^3+68n^2+52n+9.$$
Since $\Delta_1$ is quadratic in $M$ it follows that inequality $\Delta_1<0$ is fulfilled if $M$ is between roots $M_1$, $M_2$ of $\Delta_1$ where
$$M_{1,2}=
\frac{104n^4+88n^3-132n^2-104n-18}{720n^4+576n^3+416n^2+208n+36}+$$
$$+\frac{\pm 4\sqrt{2}n\sqrt{1418n^6+2156n^5-946n^4-3068n^3-1905n^2-485n-45}}{720n^4+576n^3+416n^2+208n+36}$$
tend to $$\frac{104\pm 4\sqrt{2\cdot 1418}}{720}=-0.151,\;0.440$$ when $n$ tends to $\infty$. Since $k\geq 5$ and $\frac{3n-j}{4n-j}<\frac{3}{4}$, $j>0$ it follows $0<M<\frac{3^3}{4^3}=0.422$ so that M is really between $M_1$ and $M_2$ when $n$ is large enough. If we replace $k=3n+1$ in $e_k$ and in $G_1$ we get
$$e_{3n+1}=G_1(13n^3-7n^2-10n-2)+2(6n^2-5n-2)$$ which is $>0$ when $n\geq 2$. Since $(n-1)(4n+2)(4n+1)-n(4n+1)(4n-k+2)>0$ is equivalent with $k>4+2/n$, which is obviously true for $n\geq 3$, it follows that $e_k>0$ is also fulfilled for $k=3n+2,3n+3,\ldots,4n$. Finally we completed the proof that $e_k>0$ for $k\geq 5$ and $n$ large enough.

It remains to be proved that $f_k>0$ when $4n\geq k\geq 5$.
$$f_k=6n\binom{2n+2}{k}-6n\binom{2n+1}{k}+6n\binom{2n}{k}+(4n+2)\binom{n+2}{k}-12n\binom{n+1}{k}+$$
$$+(4n-2)\binom{n}{k}+(n+1)\binom{2}{k}-n\binom{1}{k}+(n-1)\binom{0}{k}\geq 0 $$
Let us prove that $f_{k}>0$, $n \geq k\geq 5$. Since
\begin{align}\label{intro:fk}
f_k & =H_1[6n(2n+2)(2n+1)-6n(2n+1)(2n-k+2)+6n(2n-k+2)(2n-k+1)]+\\
& + H_2[(4n+2)(n+2)(n+1)-12n(n+1)(n-k+2)+(4n-2)(n-k+2)(n-k+1)]
\end{align}
where $H_1=\frac{2n(2n-1)\cdots (2n-k+3)}{k!}$ and
$H_2=\frac{n(n-1)\cdots (n-k+3)}{k!}$
so that if we divide both sides of the inequality $f_k>0$ with $H_1$ we get the equivalent inequality
$$6n(2n+2)(2n+1)-6n(2n+1)(2n-k+2)+6n(2n-k+2)(2n-k+1)+$$
$$+N[(4n+2)(n+2)(n+1)-12n(n+1)(n-k+2)+(4n-2)(n-k+2)(n-k+1)]>0$$
where $N=\frac{n(n-1)\cdots (n-k+3)}{2n(2n-1)\cdots (2n-k+3)}$
The left side of the inequality is quadratic in $k$ thus we reorder it and then we have
$$(4Nn+6n-2N)k^2+(4Nn^2-12n^2+4Nn-12n+6N)k+$$
$$-4Nn^3+24n^3+36n^2-12Nn^2-8Nn+12n.$$
so that it is fulfilled if its discriminant $\Delta_2<0$ where
$$\Delta_2=(20n^4+48n^3+24n-4n+9)N^2+(12n^2-72n^3-96n^4-12n)N-(108n^4+144n^3+36n^2)$$
Since $\Delta_2$ is quadratic in $N$ it follows that inequality $\Delta_2<0$ is fulfilled if $N$ is between roots $N_1$, $N_2$ of $\Delta_2$ where
$$N_{1,2}=\frac{48n^4+36n^3-6n^2+6n\pm 6\sqrt{2}n\sqrt{(n+1)(62n^5+98n^4+54n^3+14n^2+10n+5)}}{20n^4+48n^3+24n^2-4n+9}$$
tend to $$\frac{48\pm 6\sqrt{2\cdot 62}}{20}=-0.941,\;5.741$$ when $n$ tends to $\infty$. Since $k\geq 5$ and $\frac{n-j}{2n-j}<\frac{1}{2}$, $j>0$ it follows $0<M<\frac{1}{2^3}=0.125$ so that M is really between $M_1$ and $M_2$ when $n$ is large enough. If we replace $k=n+1$ in $f_k$ and in $H_1$ we get
$$f_{n+1}=H_1(18n^3+24n^2+6n)+2(2n^2-n+2)$$ which is $>0$ when $n\geq 2$. Since $6n\binom{2n+2}{k}-6n\binom{2n+1}{k}=6n\binom{2n+1}{k-1}>0$ for $k>0$, it follows that $f_k>0$ is also fulfilled for $k=n+2,n+3,\ldots,2n+1$. Finally we completed the proof that $d_k=e_k+f_k>0$ for $k$ such that $4n+2 \geq k\geq 4$ and $n$ large enough.

\qed
\end{proof}
\begin{theorem}\label{sec:Limit}
The house of $P_{2n}(x)=x^{2n}-x^{n+1}-x^n-x^{n-1}+1$ raised to $2n$-th power
tends to $\alpha^2=6.854\ldots$ where $$\alpha=\frac{3+\sqrt{5}}{2}=2.618\ldots$$
is the greatest root of $x^2-3x+1$.
\end{theorem}
\begin{lemma}\label{IneqMatInd} For a real number $a>1$ and a natural number $n\ge 0$ the inequality
$$-2a^{2n+2}+a^2<-a^{n+1}-a^n+1 $$ is valid.
\end{lemma}
\begin{proof}. We rewrite the inequality as $$a^2-1<2a^{2n+2}-a^{n+1}-a^n $$ and prove it using the principle of mathematical induction. For $n=0$ we have $a^2-1<2a^{2}-a-1 $ which is true.
Assume that the inequality is true for $n=k$, $k\ge 0$ and let us prove that it is true for $n=k+1$. For $n=k$ we have $$a^2-1<2a^{2k+2}-a^{k+1}-a^k .$$ Let us multiply both sides of this inequality by $a$. Then we obtain $$a(a^2-1)<2a^{2k+3}-a^{k+2}-a^{k+1}.$$ Since $a^2-1<a(a^2-1)$ we deduce the inequality $$a^2-1<2a^{2k+3}-a^{k+2}-a^{k+1} $$ which coincides with the rewritten inequality for $n=k+1$. \qed
\end{proof}
\textit{Proof of Theorem} \ref{sec:Limit} Using the Theorem \ref{sec:RealRoot} it is suficient to prove that the sequence $\left(\alpha_n^{2n}\right)_{n\ge 1}$ tends to $\alpha^2$.
Let we denote $Q_{2n}(x)=x^{2n}-3x^n+1$. Using the inequality of arithmetic and geometric means we have
\begin{align*}
Q_{2n}(\sqrt[n]{\alpha}) & = \\
& = 0\\
& = P_{2n}(\alpha_n)\\
& = \alpha_n^{2n}-\alpha_n^{n+1}-\alpha_n^n-\alpha_n^{n-1}+1\\
& <  \alpha_n^{2n}-\alpha_n^n-2\sqrt{\alpha_n^{n+1}\alpha_n^{n-1}}+1\\
& =  \alpha_n^{2n}-3\alpha_n^n+1\\
& = Q_{2n}(\alpha_n).
\end{align*}
If $x\in [\sqrt[n]{\alpha},+\infty)$ then $Q'_{2n}(x)=2nx^{2n-1}-3nx^{n-1}=nx^{n-1}(2x^n-3)\ge  n\alpha^{(n-1)/n}(2\cdot 2.6-3)>0$ so that $Q_{2n}(x)$ increase. Hence from $Q_{2n}(\sqrt[n]{\alpha})<Q_{2n}(\alpha_n)$ we deduce that
\begin{equation}\label{LPolice}
  \sqrt[n]{\alpha}<\alpha_{n}.
\end{equation}
On the other hand
\begin{align*}
Q_{2n}(\alpha_{n+1}) & = \\
& = \alpha_{n+1}^{2n}-3\alpha_{n+1}^n+1\\
& < \alpha_{n+1}^{2}(\alpha_{n+1}^{2n}-3\alpha_{n+1}^n+1)\\
& = \alpha_{n+1}^{2n+2}-\alpha_{n+1}^{n+2}-2\alpha_{n+1}^{n+2}+\alpha_{n+1}^{2}\;\;(\text{by Lemma } \ref{IneqMatInd})\\
& <  \alpha_{n+1}^{2n+2}-\alpha_{n+1}^{n+2}-\alpha_{n+1}^{n+1}-\alpha_{n+1}^{n}+1\\
& =  0.\\
\end{align*}
Since $Q_{2n}(x)$ increase on $[\sqrt[n]{\alpha},+\infty)$ it follows that it is positive on $(\sqrt[n]{\alpha},+\infty)$. Thus from $Q_{2n}(\alpha_{n+1})<0$ we deduce that $\alpha_{n+1}<\sqrt[n]{\alpha}$.
Clearly then
\begin{equation}\label{DPolice}
  \alpha_{n}<\sqrt[n-1]{\alpha}
\end{equation}
is valid.
Raising both sides of \eqref{LPolice} and of \eqref{DPolice} to the power $2n$ we have
$$\alpha^2<\alpha_{n}^{2n}<{\alpha}^{\frac{2n}{n-1}}.$$ Finaly the claim follows by the squeeze (two policemen) theorem .
\qed

\begin{theorem}\label{sec:HexToOct1}
If $P(x)$ is an antireciprocal, primitive hexanomial of degree $d>7$ such that six of its non-zero coefficients are $1$ or $-1$ then there is a natural number $p<d$ such that $P(x)(x^p-1)$ is a reciprocal octanomial (having the house equal to the house of $P(x)$), such that eight of its non-zero coefficients are also $1$ or $-1$.
\end{theorem}
\begin{proof}
. Let $a$, $b$, $d$ be integers such that
\begin{equation}\label{intro:Ineqabd}
d-a>d-b>b>a>0.
\end{equation}
Let $P_1(x)=x^d- x^{d-a}\pm x^{d-b}\mp x^{b}+x^{a}-1$ then $$P_1(x)(x^{d-a}-1)=x^{d+d-a}- x^{d-a+d-a}\pm x^{d-b+d-a}\mp x^{b+d-a} + \underline{x^{a+d-a}}-\underline{\underline{x^{d-a}}}-(\underline{x^d}- \underline{\underline{x^{d-a}}}\pm x^{d-b}\mp x^{b}+x^{a}-1)$$ i.e.
$P_1(x)(x^{d-a}-1)=x^{2d-a}- x^{2d-2a}\pm x^{2d-b-a}\mp x^{b+d-a} \mp x^{d-b}\pm x^{b}-x^{a}+1$ which is written so that exponents of its eight monomials are strictly decreasing. Therefore it is reciprocal and has exactly eight monomials.

Let $P_2(x)=x^d+ x^{d-a}- x^{d-b}+ x^{b}-x^{a}-1$ and $b\ne 2a$ then $$P_2(x)(x^{d-b}-1)=x^{d+d-b}+ x^{d-a+d-b}- x^{d-b+d-b}+ \underline{x^{b+d-b}} - x^{a+d-b}-\underline{\underline{x^{d-b}}}-(\underline{x^d}+ x^{d-a}- \underline{\underline{x^{d-b}}}+ x^{b}-x^{a}-1)$$ that is
$P_2(x)(x^{d-b}-1)=x^{d+d-b}+ x^{d-a+d-b}- x^{d-b+d-b} - x^{a+d-b}- x^{d-a}- x^{b}+x^{a}+1$. Since $2a\ne b$ it follows that either $a+d-b>d-a$ or $a+d-b<d-a$. In the first case $P_2(x)(x^{d-b}-1)$ is written so that exponents of its eight monomials are strictly decreasing. In the second case we get such polynomial if we transpose middle two monomials. 

If $b=2a$ then $P(x)=x^d+ x^{d-a}- x^{d-2a}+ x^{2a}-x^{a}-1$. Since $d-b>b$ it follows that $d> 4a$ then for $p=d-3a$ $P(x)(x^{d-3a}-1)=$ $$=x^{d+d-3a}+ x^{d-a+d-3a}- x^{d-2a+d-3a}+ \underline{x^{2a+d-3a}} - \underline{\underline{x^{a+d-3a}}}-x^{d-3a}-(x^d+ \underline{x^{d-a}}- \underline{\underline{x^{d-2a}}}+ x^{2a}-x^{a}-1)$$ is equal to
$x^{d+d-3a}+ x^{d-a+d-3a}- x^{d-2a+d-3a}-x^d-x^{d-3a}- x^{2a}+x^{a}+1$. Since $P(x)$ is primitive and $d>5$ it follows that $d\ne 5a$. Using $d>4a$ we conclude that in $P(x)(x^{d-3a}-1)$ none of two monomials has exponents equal to each other so that $P(x)(x^{d-3a}-1)$ is reciprocal and has exactly eight monomials.


Let $P_3(x)=x^d+ x^{d-a}+ x^{d-b}- x^{b}-x^{a}-1$ then $$P_3(x)(x^{b}-1)=x^{d+b}+ x^{d-a+b}+ \underline{x^{d-b+b}}- x^{b+b} - x^{a+b}-\underline{\underline{x^{b}}}-(\underline{x^d}+ x^{d-a}+ x^{d-b}- \underline{\underline{x^{b}}}-x^{a}-1)$$ is equal to
$x^{d+b}+ x^{d-a+b}- x^{2b} - x^{a+b}- x^{d-a}- x^{d-b}+x^{a}+1$. Since, using \autoref{intro:Ineqabd}, ${d+b}>{d-a+b}>{d-a}> {d-b}>{a}>1$ and ${d+b}>{d-a+b}> {2b} >{a+b}>{a}>1$ it follows that all pairs of eventually equal exponents are: $(2b,d-a)$, $(2b,d-b)$, $(a+b,d-a)$, $(a+b,d-b)$.
Finally we conclude that if
\begin{equation}\label{intro:IneqP3b}
d\ne 2a+b,\;\; d\ne a+2b,\;\; d\ne 3b
\end{equation}
then $P_3(x)(x^{b}-1)$ is an octanomial.

If we take $p=b-a$ then $$P_3(x)(x^{b-a}-1)=x^{d+b-a}+ x^{d-a+b-a}+ \underline{x^{d-b+b-a}}- x^{b+b-a} - \underline{\underline{x^{a+b-a}}}-x^{b-a}-(x^d+ \underline{x^{d-a}}+ x^{d-b}- \underline{\underline{x^{b}}}-x^{a}-1)$$ is equal to
$x^{d+b-a}+ x^{d+b-2a}-x^d - x^{2b-a}- x^{d-b}-x^{b-a}+x^{a}+1$. Since, using \autoref{intro:Ineqabd}, $d+b-a>\max(d+b-2a,d)\ge \min(d+b-2a,d)>\max(2b-a,d-b)\ge \min(2b-a,d-b)>\max(b-a,a)\ge \min(b-a,a)>1$ it follows that all pairs of eventually equal exponents are: $(d+b-2a,d)$, $(2b-a,d-b)$, $(b-a,a)$.
Therefore if $d\ne 3b-a$, $b\ne 2a$ then $P_3(x)(x^{b-a}-1)$ is an octanomial.

If $d=3b-a$ and $b\ne 2a$ then $d \ne a+2b$ so that if $d\ne b+2a$ then all conditions \autoref{intro:IneqP3b} for $P_3(x)(x^{b}-1)$ to be an octanomial are fulfilled.
If $d=3b-a$ and $d= b+2a$ it follows that $2b=3a$, $b=1.5a$ so that $P_3(x)=x^{3.5a}+x^{2.5a}+x^{2a}-x^{1.5a}-x^a-1$. We conclude that $a$ is even i.e. $a=2a_1$ so that $P_3(x)=x^{7a_1}+x^{5a_1}+x^{4a_1}-x^{3a_1}-x^{2a_1}-1$ and it is either of degree seven for $a_1=1$ or not primitive for $a_1>1$.

If $d=3b-a$ and $b=2a$ then $P_3(x)=x^{5a}+x^{4a}+x^{3a}-x^{2a}-x^a-1$ is either of degree five or is not primitive.
\qed
\end{proof}



We created a procedure which generate all primitive reciprocal and antireciprocal polynomials of degree $D$ with at most eight monomials. Then we use the standard procedures to find all roots of the polynomial, the root $r_{max}$ with maximal modulus and for factoring the polynomial. Consequently, the degree $d$ of $r_{max}$ is determined so that we are able to decide whether $r_{max}$ should be inserted in the list of $d$-th degree algebraic integers with small house.
Finally for all even $d\le 180$ we determine the smallest value of $\house{\alpha}$ for reciprocal $\alpha$ having a primitive minimal polynomial $R_d(x)$.
For $D\approx 100$ the computing took ten minutes while for $D\approx 200$ it spent two hours and the half on a 3.7 Ghz PC. So the whole calculation has taken about seventy days.

\section{Results}

In Table 1 we present the smallest house, $\mathrm{mrp}(d)$, of monic, irreducible, reciprocal, \textbf{primitive}, noncyclotomic  polynomials with integer coefficients of even degree $d$, each of them is a factor of a reciprocal or antireciprocal $D$-th degree polynomial with at most eight monomials, for $d$ at most 180, $D\le 1.5d$ and $D\le 210$.
The minimum $\mathrm{mrp}(d)$ is attained for a polynomial $R_d(x)$ with $\nu$ conjugates outside the unit disc. A denominator is a product of cyclotomic polynomials $\Phi_n$.
\begin{small}
\begin{longtable}{rllllll} 
\caption{The smallest house found of irreducible, reciprocal \textbf{primitive} algebraic integers}\\
\hline\noalign{\smallskip}
$d$ &	$\nu_r$ &	$\mathrm{mrp}(d)$ & $R_d(x)$ \\ [0.5ex] 
\noalign{\smallskip}\hline\noalign{\smallskip}
 2 & 1 & 2.618033989 &  1  3  \\
 4 & 2 & 1.539222338 &  1  1  3 \\
 6 & 2 & 1.321663156 &  1  2  2  1 \\
 8 & 2 & 1.169283030 &  1  0  0 -1  1 \\
10 & 2 & 1.125714822 & $(x^{15} + x^{13} - x^{12} - x^{3} + x^{2} + 1)/(\Phi_2\Phi_{10})$ \\
12 & 2 & 1.108054854 & $(x^{18} + x^{11} - x^{9} + x^{7} + 1)/(\Phi_3\Phi_{12}) $ \\
14 & 4 & 1.093901686 & $(x^{19} + x^{15} - x^{12} - x^{7} + x^{4} + 1)/(\Phi_2\Phi_{10})$ \\
16 & 4 & 1.085689416 & $(x^{20} - x^{19} + x^{13} - x^{10} + x^{7} - x + 1)/\Phi_{12}$\\
18 & 4 & 1.071850721 & $(x^{27} + x^{18} - x^{16} - x^{11} + x^{9} + 1)/(\Phi_2\Phi_{15})$ \\
20 & 4 & 1.060442046 & $(x^{32} + x^{25} + x^{21} - x^{11} - x^{7} - 1)/(\Phi_1\Phi_2\Phi_3\Phi_5\Phi_8)$\\
22 & 4 & 1.066217585 & $(x^{29} + x^{25} - x^{17} - x^{12} + x^{4} + 1)/(\Phi_2\Phi_{6}\Phi_{10})$ \\
24 & 4 & 1.060034246 & $(x^{31} - x^{26} + x^{16} + x^{15} - x^{5} + 1)/(\Phi_2\Phi_{6}\Phi_{10})$ \\
26 & 8 & 1.057848469 & $(x^{37} + x^{25} + x^{24} - x^{23} - x^{14} + x^{13} + x^{12} + 1)/(\Phi_2\Phi_{10}\Phi_{18})$\\
28 & 8 & 1.047786891 & $(x^{40} + x^{33} + x^{21} - x^{19} - x^{7} - 1)/(\Phi_1\Phi_{2}\Phi_8\Phi_{18})$ \\
30 & 6 & 1.049786124 & $(x^{39} - x^{32} + x^{21} + x^{18} - x^{7} + 1)/(\Phi_2\Phi_{15})$ \\
32 & 8 & 1.048455379 & $(x^{45} + x^{34} + x^{28} + x^{23} + x^{22} + x^{17} + x^{11} + 1)/(\Phi_2\Phi_{4}\Phi_8\Phi_{18})$\\
34 & 10 & 1.047503370 & $(x^{42} - x^{28} + x^{25} - x^{17} + x^{14} - 1)/(\Phi_1\Phi_2\Phi_3\Phi_8)$ \\
36 & 8 & 1.045445455 & $x^{43} + x^{35} - x^{24} - x^{19} + x^8 + 1)/(\Phi_2\Phi_{10})$ \\
38 & 12 & 1.043402608 & $(x^{47} - x^{32} + x^{30} + x^{28} + x^{19} + x^{17} - x^{15} + 1)/(\Phi_2\Phi_{4}\Phi_{18})$\\
40 & 10 & 1.041409418 & $(x^{47} - x^{39} + x^{24} + x^{23} - x^8 + 1)/(\Phi_2\Phi_{18})$ \\
42 & 8 & 1.038052321 & $(x^{51} - x^{43} + x^{27} + x^{24} - x^8 + 1)/(\Phi_2\Phi_{15})$ \\
44 & 14 & 1.038300334 & $(x^{53} - x^{38} + x^{31} + x^{30} + x^{23} + x^{22} - x^{15} + 1)/(\Phi_{2}\Phi_{4}\Phi_{18})$\\
46 & 10 & 1.034973093 & $(x^{56} + x^{45} + x^{33} - x^{23} - x^{11} - 1)/(\Phi_1\Phi_{2}\Phi_5\Phi_{8})$ \\
48 & 10 & 1.033839781 & $(x^{57} + x^{46} + x^{33} - x^{24} - x^{11} - 1)/(\Phi_1\Phi_{30})$ \\
50 & 10 & 1.031791233 & $(x^{61} - x^{50} + x^{33} + x^{28} - x^{11} + 1)/(\Phi_2\Phi_{6}\Phi_{10}\Phi_{12})$ \\
52 & 12 & 1.030630825 & $(x^{64} + x^{51} + x^{39} - x^{25} - x^{13} - 1)/(\Phi_1\Phi_{2}\Phi_{8}\Phi_{12})$ \\
54 & 12 & 1.030648009 & $(x^{64} + x^{55} - x^{37} + x^{27} - x^9 - 1)/(\Phi_1\Phi_{2}\Phi_{5}\Phi_{8})$ \\
56 & 12 & 1.030259738 & $(x^{64} + x^{53} + x^{33} - x^{31} - x^{11} - 1)/(\Phi_1\Phi_{2}\Phi_{3}\Phi_{8})$ \\
58 & 12 & 1.029612538 & $(x^{66} + x^{53} - x^{39} - x^{27} + x^{13} + 1)/\Phi_{30}$ \\
60 & 12 & 1.028423299 & $x^{68} + x^{57} - x^{35} + x^{33} - x^{11} - 1)/(\Phi_1\Phi_2\Phi_{18})$ \\
62 & 12 & 1.028644239 & $(x^{67} + x^{55} - x^{36} - x^{31} + x^{12} + 1)/(\Phi_2\Phi_{10})$ \\
64 & 18 & 1.026826118 & $(x^{97} + x^{68} - x^{63} - x^{50} - x^{47} - x^{34} + x^{29} + 1)/(\Phi_1^2\Phi_2\Phi_3\Phi_5\Phi_7\Phi_9\Phi_{21})$\\
66 & 14 & 1.026395809 & $(x^{73} + x^{60} - x^{39} - x^{34} + x^{13} + 1)/(\Phi_2\Phi_{18})$ \\
68 & 20 & 1.024213262 & $(x^{83} + x^{59} + x^{48} - x^{46} - x^{37} + x^{35} + x^{24} + 1)/(\Phi_2\Phi_{15}\Phi_{18})$\\
70 & 18 & 1.025005536 & $(x^{78} - x^{52} + x^{43} - x^{35} + x^{26} - 1)/(\Phi_1\Phi_{2}\Phi_{3}\Phi_{8})$ \\
72 & 14 & 1.023289256 & $(x^{84} + x^{67} - x^{51} - x^{33} + x^{17} + 1)/(\Phi_8\Phi_{30})$ \\
74 & 16 & 1.023505081 & $(x^{86} + x^{75} - x^{53} + x^{33} - x^{11} - 1)/(\Phi_1\Phi_{2}\Phi_{5}\Phi_{18})$ \\
76 & 16 & 1.022682125 & $(x^{88} + x^{69} + x^{57} - x^{31} - x^{19} - 1)/(\Phi_1\Phi_{2}\Phi_{8}\Phi_{18})$ \\
78 & 16 & 1.022207266 & $(x^{87} - x^{71} + x^{48} + x^{39} - x^{16} + 1)/(\Phi_2\Phi_{15})$ \\
80 & 16 & 1.020969200 & $(x^{92} + x^{75} + x^{51} - x^{41} - x^{17} - 1)/(\Phi_1\Phi_{2}\Phi_{5}\Phi_{18})$ \\
82 & 18 & 1.021813323 & $(x^{89} - x^{75} + x^{47} + x^{42} - x^{14} + 1)/(\Phi_2\Phi_{18})$ \\
84 & 16 & 1.020986553 & $(x^{93} - x^{79} + x^{51} + x^{42} - x^{14} + 1)/(\Phi_2\Phi_{15})$ \\
86 & 15 & 1.021181880 & $(x^{92} + x^{49} - x^{46} + x^{43} + 1)/(\Phi_3\Phi_{12})$ \\
88 & 16 & 1.020725627 & $(x^{96} - x^{83} + x^{57} + x^{39} - x^{13} + 1)/(\Phi_{15})$ \\
90 & 18 & 1.019367563 & $(x^{99} - x^{82} + x^{51} + x^{48} - x^{17} + 1)/(\Phi_2\Phi_{15})$ \\
92 & 18 & 1.018537555 & $(x^{104} + x^{85} + x^{57} - x^{47} - x^{19} - 1)/(\Phi_1\Phi_{2}\Phi_{3}\Phi_{5}\Phi_{8})$ \\
94 & 22 & 1.019028397 & $(x^{102} - x^{68} + x^{55} - x^{47} + x^{34} - 1)/(\Phi_1\Phi_{2}\Phi_{6}\Phi_{8})$ \\
96 & 20 & 1.017826909 & $(x^{108} + x^{89} - x^{57} - x^{51} + x^{19} + 1)/(\Phi_8\Phi_{30})$ \\
98 & 20 & 1.018158761 & $(x^{106} + x^{87} + x^{57} - x^{49} - x^{19} - 1)/(\Phi_1\Phi_{2}\Phi_{18})$ \\
100 & 20 & 1.017325179 & $(x^{112} + x^{95} - x^{61} + x^{51} - x^{17} - 1)/(\Phi_1\Phi_{2}\Phi_{3}\Phi_{5}\Phi_{8})$ \\
102 & 20 & 1.017650742 & $(x^{110} + x^{93} - x^{59} + x^{51} - x^{17} - 1)/(\Phi_1\Phi_{2}\Phi_{18})$ \\
104 & 18 & 1.017330618 & $(x^{112} + x^{59} + x^{56} + x^{53} + 1)/(\Phi_5\Phi_{12})$ \\
106 & 20 & 1.016982801 & $(x^{113} - x^{94} + x^{57} + x^{56} - x^{19} + 1)/(\Phi_2\Phi_6\Phi_{12})$ \\
108 & 24 & 1.016504509 & $(x^{117} + x^{78} - x^{61} - x^{56} + x^{39} + 1)/(\Phi_2\Phi_{15})$ \\
110 & 20 & 1.016401551 & $(x^{118} + x^{95} + x^{69} - x^{49} - x^{23} - 1)/(\Phi_1\Phi_{2}\Phi_{3}\Phi_{5})$ \\
112 & 24 & 1.015669389 & $(x^{124} + x^{105} - x^{67} + x^{57} - x^{19} - 1)/(\Phi_1\Phi_{2}\Phi_{5}\Phi_{18})$ \\
114 & 26 & 1.015479967 & $(x^{126} + x^{84} + x^{73} + x^{53} + x^{42} + 1)/(\Phi_8\Phi_{30})$ \\
116 & 24 & 1.015625196 & $(x^{128} + x^{111} - x^{77} + x^{51} - x^{17} - 1)/(\Phi_1\Phi_{2}\Phi_8\Phi_{18})$ \\
118 & 24 & 1.014982538 & $(x^{128} + x^{105} + x^{69} - x^{59} - x^{23} - 1)/(\Phi_1\Phi_{2}\Phi_{5}\Phi_{8})$ \\
120 & 24 & 1.014911998 & $(x^{129} - x^{107} - x^{66} + x^{63} + x^{22} - 1)/(\Phi_1\Phi_{30})$ \\
122 & 24 & 1.014416023 & $(x^{133} - x^{110} + x^{69} + x^{64} - x^{23} + 1)/(\Phi_2\Phi_{6}\Phi_{10}\Phi_{12})$ \\
124 & 24 & 1.014722774 & $(x^{131} - x^{106} + x^{75} + x^{56} - x^{25} + 1)/(\Phi_2\Phi_6\Phi_{12})$ \\
126 & 26 & 1.014273084 & $(x^{136} + x^{115} - x^{73} + x^{63} - x^{21} - 1)/(\Phi_1\Phi_{2}\Phi_{5}\Phi_{8})$ \\
128 & 24 & 1.014122887 & $(x^{136} + x^{113} + x^{69} - x^{67} - x^{23} - 1)/(\Phi_1\Phi_{2}\Phi_{3}\Phi_{8})$ \\
130 & 30 & 1.014012153 & $(x^{138} - x^{92} + x^{73} - x^{65} + x^{46} - 1)/(\Phi_1\Phi_{2}\Phi_{6}\Phi_{8})$ \\
132 & 38 & 1.013640050 & $(x^{149} + x^{107} - x^{88} + x^{84} + x^{65} - x^{61} + x^{42} + 1)/(\Phi_2\Phi_4\Phi_{15}\Phi_{18})$\\
134 & 26 & 1.013498976 & $(x^{144} + x^{115} + x^{87} - x^{57} - x^{29} - 1)/(\Phi_1\Phi_{2}\Phi_{5}\Phi_{8})$ \\
136 & 30 & 1.013428352 & $(x^{144} + x^{96} - x^{77} - x^{67} + x^{48} + 1)/(\Phi_{15})$ \\
138 & 26 & 1.013055754 & $(x^{147} - x^{121} + x^{78} + x^{69} - x^{26} + 1)/(\Phi_2\Phi_{15})$ \\
140 & 28 & 1.012795821 & $(x^{152} + x^{129} - x^{83} + x^{69} - x^{23} - 1)/(\Phi_1\Phi_{2}\Phi_{8}\Phi_{18})$ \\
142 & 28 & 1.012635977 & $(x^{152} + x^{125} + x^{81} - x^{71} - x^{27} - 1)/(\Phi_1\Phi_{2}\Phi_{5}\Phi_{8})$ \\
144 & 28 & 1.012523350 & $(x^{156} - x^{133} + x^{87} + x^{69} - x^{23} + 1)/(\Phi_8\Phi_{15})$ \\
146 & 28 & 1.012486423 & $(x^{154} + x^{125} + x^{87} - x^{67} - x^{29} - 1)/(\Phi_1\Phi_{2}\Phi_{3}\Phi_{5})$ \\
148 & 30 & 1.012389544 & $(x^{155} - x^{129} + x^{78} + x^{77} - x^{26} + 1)/(\Phi_2\Phi_{18})$ \\
150 & 34 & 1.011974270 & $(x^{162} + x^{108} + x^{91} + x^{71} + x^{54} + 1)/(\Phi_8\Phi_{30})$ \\
152 & 30 & 1.012000260 & $(x^{160} + x^{131} + x^{87} - x^{73} - x^{29} - 1)/(\Phi_1\Phi_{2}\Phi_{3}\Phi_{8})$ \\
154 & 30 & 1.011896494 & $(x^{162} - x^{131} + x^{93} + x^{69} - x^{31} + 1)/(\Phi_{15})$ \\
156 & 30 & 1.011777895 & $(x^{163} - x^{134} + x^{87} + x^{76} - x^{29} + 1)/(\Phi_2\Phi_6\Phi_{12})$ \\
158 & 30 & 1.011697187 & $(x^{164} + x^{135} + x^{87} - x^{77} - x^{29} - 1)/(\Phi_1\Phi_{2}\Phi_{5})$ \\
160 & 30 & 1.011417661 & $(x^{168} + x^{139} - x^{87} - x^{81} + x^{29} + 1)/(\Phi_{30})$ \\
162 & 30 & 1.011256404 & $(x^{171} - x^{143} + x^{87} + x^{84} - x^{28} + 1)/(\Phi_2\Phi_{15})$ \\
164 & 32 & 1.010886959 & $(x^{176} + x^{145} + x^{93} - x^{83} - x^{31} - 1)/(\Phi_1\Phi_{2}\Phi_{3}\Phi_{5})$ \\
166 & 32 & 1.011082816 & $(x^{173} - x^{142} + x^{93} + x^{80} - x^{31} + 1)/(\Phi_2\Phi_6\Phi_{12})$ \\
168 & 32 & 1.010841122 & $(x^{177} - x^{146} + x^{93} + x^{84} - x^{31} + 1)/(\Phi_2\Phi_{15})$ \\
170 & 32 & 1.011010457 & $(x^{178} - x^{141} - x^{111} + x^{67} + x^{37} - 1)/(\Phi_1\Phi_{2}\Phi_{9})$ \\
172 & 34 & 1.010501552 & $(x^{184} + x^{155} - x^{97} + x^{87} - x^{29} - 1)/(\Phi_1\Phi_{2}\Phi_{3}\Phi_{5}\Phi_{8})$ \\
174 & 34 & 1.010497511 & $(x^{183} - x^{149} + x^{102} + x^{81} - x^{34} + 1)/(\Phi_2\Phi_{15})$ \\
176 & 34 & 1.010469386 & $(x^{184} + x^{149} + x^{105} - x^{79} - x^{35} - 1)/(\Phi_1\Phi_{2}\Phi_{3}\Phi_{8})$ \\
178 & 38 & 1.010370370 & $(x^{186} - x^{124} + x^{97} - x^{89} + x^{62} - 1)/(\Phi_1\Phi_{2}\Phi_{3}\Phi_{8})$ \\
180 & 34 & 1.010150047 & $(x^{189} - x^{157} + x^{96} + x^{93} - x^{32} + 1)/(\Phi_2\Phi_{15})$ \\

\noalign{\smallskip}\hline
\end{longtable}
\end{small}
Assuming that Table 1 contains $\mathrm{mrp}(d)$ for every $d$, we create Table 2 using the following algorithm based on the corollary \ref{sec:compositeRec}:
\begin{enumerate}
\item We calculate $\mathrm{mrp}^d(d)$ and write it in the second column.
\item For every even divisor $d_j$ of $d$ we calculate $\mathrm{mrp}^{d_j}(d_j)$ then find their minimum and write it in the third column. Let the minimum be attained for $d_j=d_0$.
\item For $k_0=d/d_0$ we calculate $\mathrm{mr}(d)=\sqrt[k_0]{\mathrm{mrp}(d_0)}$ and write it in the fourth column. We can also calculate $\mathrm{mr}(d)$ as the $d$-th root of the minimum written in the third column.
\item The minimal polynomial $P_d(x)$ of the extremal reciprocal algebraic integer $\alpha$ whose house is denoted
by $\mathrm{mr}(d)$ is equal to $R_{d_0}(x^{k_0})$. If $d=d_0$ so that $P_d(x)$ is primitive then we present the first half coefficients of $P_d(x)$ in the sixth column.
\item We calculate the number $\nu(d)$ of roots of $P_d(x)$ outside the unit disc as $\nu(d))=k_0\nu_r(d_0)$ and write it in the fifth column.
\end{enumerate}

\begin{small}
\begin{longtable}{rllllll} 
\caption{The smallest values found of $\house{\alpha}$ for reciprocal $\alpha$ having a minimal polynomial $P_d(x)$ of even degree $d\le 180$. }\\
\hline\noalign{\smallskip}
$d$ & $\mathrm{mrp}^d(d)$ & $\min_{\delta \mid d}(\mathrm{mrp}^\delta(\delta))$ & $\mathrm{mr}(d)$ & $\nu$ &	$P_d(x)$ \\ [0.5ex] 
\noalign{\smallskip}\hline\noalign{\smallskip}
2 & 6.854101968 & 6.854101968 & 2.618033989 & 1 & 1 3\\
4 & 5.613133701 & 5.613133701 & 1.539222338 & 2 & 1 1 3\\
6 & 5.329970273 & 5.329970273 & 1.321663156 & 2 & 1 2 2 1\\
8 & 3.494275747 & 3.494275747 & 1.169283030 & 2 & 1 0 0 $-1$ 1\\
10 & 3.268013514 & 3.268013514 & 1.125714822 & 2 & 1 0 1 1 0 1\\
12 & 3.425587986 & 3.425587986 & 1.108054854 & 2 & 1 1 1 0 $-1$ $-1$ $-1$\\
14 & 3.513145071 & 3.513145071 & 1.093901686 & 4 & 1 0 0 0 1 1 0 1\\
16 & 3.726401663 & 3.494275747 & 1.081333912 & 4 & $R_{8}(x^2)$\\
18 & 3.486723207 & 3.486723207 & 1.071850720 & 4 & 1 0 1 1 1 2 1 2 2 1\\
20 & 3.233990794 & 3.233990794 & 1.060442046 & 4 & 1 2 2 1 $-1$ $-3$ $-3$ $-2$ 0 2 3\\
22 & 4.098344884 & 4.098344884 & 1.066217585 & 4 & 1 $-1$ 0 1 0 0 0 0 0 1 0 $-1$\\
24 & 4.052075275 & 3.425587986 & 1.052641845 & 4 & $R_{12}(x^2)$\\
26 & 4.315290210 & 4.315290210 & 1.057848469 & 8 & 1 0 0 1 0 $-1$ 0 0 $-1$ $-1$ 1 0 0 2\\
28 & 3.695242104 & 3.513145071 & 1.045897550 & 8 & $R_{14}(x^2)$\\
30 & 4.295609952 & 3.268013514 & 1.040262145 & 6 & $R_{10}(x^3)$\\
32 & 4.545675907 & 3.494275747 & 1.039872065 & 8 & $R_{8}(x^4)$\\
34 & 4.844897357 & 4.844897357 & 1.047503370 & 10 & 1 1 1 0 $-1$ $-1$ 0 1 2 1 0 $-1$ $-1$ \\
& \multicolumn{5}{l}{0 0 0 0 $-1$}\\
36 & 4.952786876 & 3.425587986 & 1.034793646 & 6 & $R_{12}(x^3)$\\
38 & 5.025425981 & 5.025425981 & 1.043402608 & 12 & 1	$-1$	0	1	0	$-1$	0	1	0	-2	1	1	$-1$\\
& \multicolumn{5}{l}{$-1$	1	0	0	0	0	1}\\
40 & 5.068273424 & 3.233990794 & 1.029777668 & 8 & $R_{20}(x^2)$\\
42 & 4.799635323 & 3.513145071 & 1.030368953 & 12 & $R_{14}(x^3)$\\
44 & 5.226509253 & 4.098344884 & 1.032578125 & 8 & $R_{22}(x^2)$\\
46 & 4.861124291 & 4.861124291 & 1.034973093 & 10 & 1	$-1$	1	$-1$	0	1	$-1$	1	0	$-1$	2	$-1$\\
& \multicolumn{5}{l}{0	1	$-2$	2	0	$-1$    2	$-2$	1	1	$-2$	3}\\
48 & 4.940324629 & 3.425587986 & 1.025983355 & 8 & $R_{12}(x^4)$\\
50 & 4.781802892 & 3.268013514 & 1.023966331 & 10 & $R_{10}(x^5)$\\
52 & 4.801341912 & 4.315290210 & 1.028517608 & 16 & $R_{26}(x^2)$\\
54 & 5.104578724 & 3.486723207 & 1.023398481 & 12 & $R_{18}(x^3)$\\
56 & 5.309049530 & 3.494275747 & 1.022592979 & 14 & $R_{8}(x^7)$\\
58 & 5.433525446 & 5.433525446 & 1.029612538 & 12 & 1	1	1	0	0	$-1$	$-1$	$-1$	0	0	0	0	0	\\
& \multicolumn{5}{l}{$-1$	$-1$	0	1	1	1	1	0	$-1$	$-1$	0	0	0	0	1	0	0}\\
60 & 5.374207407 & 3.233990794 & 1.019754537 & 12 & $R_{20}(x^3)$\\
62 & 5.760262620 & 5.760262620 & 1.028644239 & 12 & 1	0	0	0	0	1	0	0	0	0	1	0	1	0	0	\\
& \multicolumn{5}{l}{1	0	1	0	0	1	0	1	0	0	1	0	1	0	0	1	1}\\
64 & 5.442545008 & 3.494275747 & 1.019741176 & 16 & $R_{8}(x^8)$\\
66 & 5.581891922 & 4.098344884 & 1.021602500 & 12 & $R_{22}(x^3)$\\
68 & 5.087997146 & 4.844897357 & 1.023476121 & 20 & $R_{34}(x^2)$\\
70 & 5.634232571 & 3.268013514 & 1.017060791 & 14 & $R_{10}(x^7)$\\
72 & 5.246695122 & 3.425587986 & 1.017248075 & 12 & $R_{12}(x^6)$\\
74 & 5.580334344 & 5.580334344 & 1.023505081 & 16 & 1	$-1$	1	0	0	1	$-1$	1	0	$-1$	2	$-1$	\\
& \multicolumn{5}{l}{0	1	$-1$	2	$-1$	0	2	$-2$	2	0	$-1$	2	$-2$	2	0	$-2$	3	$-2$	1	1	$-2$	2	$-1$	0	2	$-3$}\\
76 & 5.499086530 & 5.025425981 & 1.021470806 & 24 & $R_{38}(x^2)$\\
78 & 5.546757323 & 4.315290210 & 1.018922503 & 24 & $R_{26}(x^3)$\\
80 & 5.260309040 & 3.233990794 & 1.014779616 & 16 & $R_{20}(x^4)$\\
82 & 5.867701324 & 5.867701324 & 1.021813323 & 18 & 1	$-1$	1	0	0	0	0	0	0	$-1$	1	$-1$	\\
& \multicolumn{5}{l}{0	0	$-1$	1	$-1$	0	1	$-1$	1	0	0	1	$-1$	1	0	$-1$	1	$-1$	0	0	$-1$	1	$-1$	0	1	$-1$	1	0	0	1}\\
84 & 5.723765933 & 3.425587986 & 1.014765969 & 14 & $R_{12}(x^{7})$\\
86 & 6.065499920 & 6.065499920 & 1.021181880 & 15 & 1	1	1	0	$-1$	$-1$	$-1$	0	0	0	0	0\\	
& \multicolumn{5}{l}{	1	1	1	0   $-1$	$-1$	$-1$	0	0	0	0	0	1	1	1	0	$-1$	$-1$	$-1$	0	0	0	0	0	1	1	1	0	$-1$	$-1$	$-1$	$-1$}\\
88 & 6.081260895 & 3.494275747 & 1.014318889 & 22 & $R_{8}(x^{11})$\\
90 & 5.620473341 & 3.268013514 & 1.013244523 & 18 & $R_{10}(x^{9})$\\
92 & 5.418615039 & 4.861124291 & 1.017336273 & 20 & $R_{46}(x^{2})$\\
94 & 5.881809242 & 5.881809242 & 1.019028397 & 22 & 1	$-1$	1	0	$-1$	1	0	$-1$	2	$-1$	0	1	\\
& \multicolumn{5}{l}{$-1$	0	1	$-1$	1	0	0	0	0	0	0	0	1	$-1$	1	0	$-1$	1	0	$-1$	2	$-1$	$-1$	2	$-2$	0	2	$-2$	1	1	$-2$	1	0	$-1$	1	$-1$}\\
96 & 5.453773907 & 3.425587986 & 1.012908365 & 16 & $R_{12}(x^{8})$\\
98 & 5.833366397 & 3.513145071 & 1.012904096 & 28 & $R_{14}(x^{7})$\\
100 & 5.571592570 & 3.233990794 & 1.011806320 & 20 & $R_{20}(x^{5})$\\
102 & 5.957620928 & 4.844897357 & 1.015590141 & 30 & $R_{34}(x^{3})$\\
104 & 5.971177596 & 3.494275747 & 1.012102711 & 26 & $R_{8}(x^{13})$\\
106 & 5.959948109 & 5.959948109 & 1.016982801 & 20 & 1	0	1	$-1$	0	$-1$	0	0	0	0	0	0	1	\\
& \multicolumn{5}{l}{0	1	$-1$	0	$-1$	0	$-1$	0	$-1$	1	0	2	0	1	$-1$	0	$-1$	0	$-1$	0	$-1$	1	0	2	0	1	$-1$	0	$-1$	0	$-1$	0	$-1$	1	0	}\\
& \multicolumn{5}{l}{2	0	1	$-1$	0	$-1$	}\\
108 & 5.858756078 & 3.425587986 & 1.011465913 & 18 & $R_{12}(x^{9})$\\
110 & 5.986666989 & 3.268013514 & 1.010823448 & 22 & $R_{10}(x^{11})$\\
112 & 5.705119145 & 3.494275747 & 1.011233395 & 28 & $R_{8}(x^{14})$\\
114 & 5.761493367 & 5.025425981 & 1.014263132 & 36 & $R_{38}(x^{3})$\\
116 & 6.040628268 & 5.433525446 & 1.014698250 & 24 & $R_{58}(x^{2})$\\
118 & 5.782442619 & 5.782442619 & 1.014982538 & 24 & 1	1	1	1	0	$-1$	$-1$	$-1$	0	1	2	2	1	\\
& \multicolumn{5}{l}{0	$-1$	$-2$	$-1$	0	1	2	2	1	0	$-2$	$-2$	$-2$	$-1$	1	2	2	2	0	$-1$	$-2$	$-2$	$-1$	0	1	2	1	1	0	$-1$	$-1$	$-1$	}\\
& \multicolumn{5}{l}{$-1$	0	0	1	1	1	1	0	$-1$	$-1$	$-2$	$-1$	0	1	1	}\\
120 & 5.907536277 & 3.233990794 & 1.009828964 & 24 & $R_{20}(x^{6})$\\
122 & 5.732766447 & 5.732766447 & 1.014416023 & 24 & 1	1	1	0	$-1$	$-2$	$-2$	$-1$	0	1	2	2	\\
& \multicolumn{5}{l}{2	1	0	$-2$	$-3$	$-3$	$-2$	0	2	3	3	1	0	$-2$	$-2$	$-2$	$-1$	0	1	2	2	1	0	$-2$	$-2$	$-2$	0	1	2	2	1	0	$-1$	$-2$	}\\
& \multicolumn{5}{l}{$-2$	$-2$	0	1	3	3	2	0	$-2$	$-3$	$-3$	$-2$	0	1	3	3	}\\
124 & 6.124611460 & 5.613133701 & 1.014009395 & 62 & $R_{4}(x^{31})$\\
126 & 5.963723281 & 3.486723207 & 1.009961690 & 28 & $R_{18}(x^{7})$\\
128 & 6.019976073 & 3.494275747 & 1.009822349 & 32 & $R_{8}(x^{16})$\\
130 & 6.103947784 & 3.268013514 & 1.009150709 & 26 & $R_{10}(x^{13})$\\
132 & 5.979385283 & 3.425587986 & 1.009371467 & 22 & $R_{12}(x^{11})$\\
134 & 6.030094351 & 6.030094351 & 1.013498976 & 26 & 1	1	1	1	0	$-1$	$-1$	$-1$	0	1	2	2	\\
& \multicolumn{5}{l}{1	0	$-1$	$-2$	$-1$	0	1	2	2	1	0	$-1$	$-1$	$-1$	0	1	1	0	0	$-1$	$-1$	0	1	1	1	0	$-1$	$-2$	$-1$	0	1	2	2	0	$-1$	}\\
& \multicolumn{5}{l}{$-2$	$-2$	$-1$	1	2	2	1	0	$-2$	$-2$	$-2$	$-1$	0	1   1	1	0	0	$-1$	$-1$	$-1$}\\
136 & 6.135568540 & 3.494275747 & 1.009241902 & 34 & $R_{8}(x^{17})$\\
138 & 5.989657253 & 4.861124291 & 1.011524376 & 30 & $R_{46}(x^{3})$\\
140 & 5.930155706 & 3.233990794 & 1.008418934 & 28 & $R_{20}(x^{7})$\\
142 & 5.948070694 & 5.948070694 & 1.012635977 & 28 & 1	$-1$	1	$-1$	0	1	$-1$	1	0	$-1$	2	\\
& \multicolumn{5}{l}{$-2$	1	0	$-1$	2	$-1$	0	1	$-2$	2	$-1$	0	1	$-1$	1	0	0	0	0	0	0	1	$-1$	1	0	$-1$	2	$-2$	1	1	$-2$	3	$-2$	0	2	}\\
& \multicolumn{5}{l}{$-3$	3	$-1$	$-1$	3	$-3$	2	0	$-2$	3	$-2$	1	1	$-2$   2	$-1$	0	1	$-1$	1	0	0	0	0	0	1}\\
144 & 6.002426057 & 3.425587986 & 1.008587168 & 24 & $R_{12}(x^{12})$\\
146 & 6.121028493 & 6.121028493 & 1.012486423 & 28 & 1	$-2$	2	$-1$	0	1	$-1$	0	1	$-1$	1	$-1$	\\
& \multicolumn{5}{l}{1	$-1$	1	0	$-1$	1	0	$-1$	2	$-2$	1	0	0	0	0	0	0	1	$-1$	0	1	$-1$	1	0	$-1$	1	0	0	0	0	0	0	1	$-1$	0	1	$-1$	1	}\\	
& \multicolumn{5}{l}{0	$-1$	1	0	0	0	0	0	0	1	$-1$	0   1	$-1$	1	0	$-1$	2	$-2$	2	$-1$	0	1	$-1$}\\
148 & 6.186604634 & 5.580334344 & 1.011684279 & 32 & $R_{74}(x^{2})$\\
150 & 5.962392616 & 3.268013514 & 1.007925793 & 30 & $R_{10}(x^{15})$\\
152 & 6.129920758 & 3.494275747 & 1.008265062 & 38 & $R_{8}(x^{19})$\\
154 & 6.179566941 & 3.513145071 & 1.008192543 & 44 & $R_{14}(x^{11})$\\
156 & 6.212825396 & 3.425587986 & 1.007924007 & 26 & $R_{12}(x^{13})$\\
158 & 6.280376894 & 6.280376894 & 1.011697187 & 30 & 1	1	1	1	1	0	0	0	0	0	1	1	1	1	\\
& \multicolumn{5}{l}{1	0	0	0	0	0	1	1	1	1	1	0	0	0	0	$-1$	0	0	0	0	1	0	0	0	0	$-1$	0	0	0	0	1	0	0	0	0	$-1$	0	0	0	0	}\\
& \multicolumn{5}{l}{1	0	0	0	0	$-1$    0	0	0	0	1	0	0	0	0	$-1$	0	0	0	0	1	0	0	$-1$	$-1$	$-2$}\\
160 & 6.150143601 & 3.233990794 & 1.007362703 & 32 & $R_{20}(x^{8})$\\
162 & 6.130955341 & 3.486723207 & 1.007739440 & 36 & $R_{18}(x^{9})$\\
164 & 5.905074844 & 5.613133701 & 1.010574477 & 82 & $R_{4}(x^{41})$\\
166 & 6.231564151 & 6.231564151 & 1.011082816 & 32 & 1	0	1	$-1$	0	$-1$	0	0	0	0	0	0	1	0	\\
& \multicolumn{5}{l}{1	$-1$	0	$-1$	0	0   0	0	0	0	1	0	1	$-1$	0	$-1$	0	$-1$	0	$-1$	1	0	2	0	1	$-1$	0	$-1$	0	$-1$	0	$-1$	1	0	2	0	1	}\\
& \multicolumn{5}{l}{$-1$	0	$-1$	0	$-1$	0	$-1$	1	0   2	0	1	$-1$	0	$-1$	0	$-1$	0	$-1$	1	0	2	0	1	$-1$	0	$-1$	0	$-1$	1	$-1$	2	$-1$}\\
168 & 6.119661348 & 3.425587986 & 1.007355930 & 28 & $R_{12}(x^{14})$\\
170 & 6.433689353 & 3.268013514 & 1.006990096 & 34 & $R_{10}(x^{17})$\\
172 & 6.030612145 & 5.613133701 & 1.010080170 & 86 & $R_{4}(x^{43})$\\
174 & 6.153655414 & 5.329970273 & 1.009663320 & 58 & $R_{6}(x^{29})$\\
176 & 6.252824150 & 3.494275747 & 1.007133998 & 44 & $R_{8}(x^{22})$\\
178 & 6.274037304 & 6.274037304 & 1.010370370 & 38 & 1	1	1	0	$-1$	$-1$	0	1	2	1	0	$-1$	$-1$	\\
&\multicolumn{5}{l}{0	1	1	1	0	0	0	0	0	0	0	1	1	1	0	$-1$	$-1$	0	1	2	1	0	$-1$	$-1$	0	1	1	1	0	0	0	0	0	0	0	1	1	1	0	$-1$	$-1$	0	}\\
&\multicolumn{5}{l}{1	1	0	$-1$	2	$-1$	0	0	0	0	0	1	1	0	$-1$	$-2$	$-1$	1	2	2	0	$-2$	$-2$	$-1$	1	2	1	0	$-1$	$-1$	0	0	0	0	$-1$}\\
180 & 6.158286769 & 3.233990794 & 1.006541955 & 36 & $R_{20}(x^{9})$\\

\noalign{\smallskip}\hline
\end{longtable}
\end{small}

\section{The old and new conjectures}

The first, fourth and sixth column of our Table 2 represent the continuation of the Table 1 of Wu and Zhang \cite{WU2017170} and these two tables definitely matches for $2 \leq d \leq 42$.
Although we can not guarantee that, for $d>42$, we have found a reciprocal polynomial with the smallest house we certainly have made a good approximation of $\mathrm{mr}(d)$. There are three reasons for our confidence.
The first one is the following
\begin{conjecture}[Wu, Zhang \cite{WU2017170}]\label{sec:CWuZhang2} Any extremal reciprocal algebraic integer $\alpha$ with degree $d \geq 6$ has minimal
polynomial which is a factor of reciprocal polynomial with at most eight monomials with
height 1.
\end{conjecture}
This conjecture is proved for $6 \leq d \leq 42$ so if it is not true for all $d$ it is reasonably to expect that it is correct for many $d$ not too large. Although we involved antireciprocal polynomials in our research and spent as many CPU time as for reciprocal polynomials, we did not disprove the conjecture. As for antireciprocal hexanomials the Theorem \ref{sec:HexToOct1} actually supports the conjecture. Using them we only succeeded to get simpler representation of many extremals in Table 1 than by using reciprocal octanomials. Also we did not find any antireciprocal octanomial such that the minimal polynomial of an extremal reciprocal algebraic integer is its factor and is not a factor of any reciprocal pentanomial, hexanomial, heptanomial or octanomial.

The second reason is our extensive computation.
We compute the minimum of the houses of all reciprocal algebraic integers of degree $d$ such that its minimal polynomial is a factor of a $D$-th degree reciprocal or antireciprocal polynomial with at most eight monomials for $d$ at most 180 and $D$ at most 210. As the factoring of a polynomial spends lot of processor time we reject a polynomial if its house is greater than $1+c_1/D$. Experimenting with several values of $c_1$ we concluded that $c_1=2.5$ is ideal. If $c_1>2.5$ then we have too much unnecessary calculations, but if $c_1<2.5$ then an extremal reciprocal can be missed. For $d \approx 200 $ the duration of computation with $c_1=2.8$ was approximately five hours, which is more than double the time spent for $c_1=2.5$ on a 3.7 Ghz PC. But our attempt to find polynomials with smaller houses increasing $c_1$ to $c_1=2.8$ failed. We discovered only few unknown polynomials with small house but no one decreased $\mathrm{mr}(d)$. Actually, many reciprocal $\alpha$ can be found in different ways, for example $1.013333049$, the subextremal reciprocal of degree $138$, as a root of the reciprocal octanomial $x^{191}-x^{168}+x^{145}+x^{115}+x^{76}+x^{46}-x^{23}+1$, is rejected by our program because it is greater than $1+2.5/191\approx 1.0131$. But this number, as a root of $x^{168} - x^{122} + x^{99} + x^{92} + x^{76} + x^{69} - x^{46} + 1$, is accepted because it is less than $1+2.5/168\approx 1.0149$.

The third reason is statistical. If we plot $$\frac{1}{\mathrm{mrp}(d)-1}$$ versus degree we can notice that these points appear to fall very close to a straight line. If we model the line using the method of least squares \cite{weisberg2013applied} then for $12\leq d \leq 40$ we get that $1/(\mathrm{mrp}(d)-1)\approx 0.51d+4.3$ and for $12 \leq d \leq 180$ we get $1/(\mathrm{mrp}(d)-1)\approx 0.52d+4.1$. Since it is almost the same line we conclude that our approximations are good. We remark that the coefficient of determination is 0.953 and 0.998 respectively, which means that there is almost perfect correlation.
Using these calculations we establish the following
\begin{conjecture}\label{sec:MLS} Let $\mathrm{mrp}(d)$ be the smallest house of monic, irreducible, reciprocal, primitive, noncyclotomic  polynomials with integer coefficients of even degree $d$. Then points $$\left(d,\frac{1}{\mathrm{mrp}(d)-1}\right)$$ are very close to a straight line. If the least squares method is used then the line of best fit through these points is $\approx 0.52d+4.1$, with the coefficient of determination close to 1.
\end{conjecture}


If we analyse our Table 2 then we conclude that it supports the next
\begin{conjecture}[Wu, Zhang \cite{WU2017170}]\label{sec:CWuZhang1} Let $\alpha$ be a reciprocal algebraic integer, not a root of unity, and let $d = \deg(\alpha) \geq 2$. Then
$$\mathrm{mr}^d(d) \geq \mathrm{mr}^{20}(20),$$
and if $10 \nmid d$ then
$$\mathrm{mr}^d(d) \geq \mathrm{mr}^{12}(12).$$
\end{conjecture}

If $p$ is a prime number then it is obvious that the minimal polynomial of the extremal reciprocal of degree $2p$ is primitive or $R_2(x^p)=x^{2p}+3x^p+1$. Table 2 
suggests that $P_8(x)
$, $P_{12}(x)
$, $P_{18}(x)
$ and $P_{20}(x)$ are the only primitive minimal polynomials of an extremal reciprocal of a degree $d$ such that $d/2$ is a composite number.

\begin{conjecture}\label{sec:compositeConjRec}
Let $d$ be an even natural number such that $d/2$ is composite. If $d\notin\{8,12,18,20\}$ then $P_d(x)$ is not primitive, where $P_d(x)$ is the minimal polynomial of an extremal reciprocal of degree $d$.
\end{conjecture}
If the previous conjecture is true then we just need to determine $\mathrm{mr}(d)$ for $d/2>10$ is a prime number. If $d/2$ is a composite number we can easily calculate $\mathrm{mr}(d)=\mathrm{mr}^{p_1/d}(p_1)$ using the algorithm. 
\begin{proposition}\label{sec:plusminusProp}
An extremal reciprocal primitive of degree $d\leq 180$ can not be a root of an reciprocal octanomial of degree $D_1$ such that $D_1<210$, $D_1<2d$ and all its inner monomials have minus sign.
An extremal reciprocal primitive of degree $d\leq 180$ can not be a root of an reciprocal octanomial of degree $D_2$ such that $D_2<210$, $D_2<1.5d$ and all its monomials have plus sign.
\end{proposition}
\begin{proof}
Analysing our list of reciprocal octanomials, which are divisible by $R_d(x)$ from Table 1, we show that the claim is true.
\end{proof}
The condition $D_1<2d$ in the previous proposition can not be omited because
for degree 86 there is the octanomial $x^{181}-x^{132}-x^{95}-x^{92}-x^{89}-x^{86}-x^{49}+1$ whose divisor $R_{86}(x)$ has the house 1.021181880 which is equal to $\mathrm{mrp}(86)$, see Table 1, but $D_1=181$ is not less than $2d=172$.
Also, the condition $D_2<1.5d$ can not be omited because
for degree 44 there is the octanomial $x^{68} + x^{46} + x^{45} + x^{37} + x^{31} + x^{23} + x^{22} + 1$ whose divisor $R_{44}(x)$ has the house 1.038300334 which is equal to $\mathrm{mrp}(44)$, see Table 1, but $D_2=68$ is not less than $1.5d=66$.
For $d=38$ there is the octanomial $x^{62} + x^{45} + x^{43} + x^{34} + x^{28} + x^{19} + x^{17} + 1$ whose divisor $R_{38}(x)$ has the house 1.043402608 which is equal to $\mathrm{mrp}(38)$, see Table 1, but $D_2=62$ is not less than $1.5\cdot38=57$ etc.
\begin{conjecture}\label{sec:plusminusConj}
An extremal reciprocal primitive of degree $d$ can not be a root, neither of an reciprocal octanomial of degree $D_1$ such that $D_1<2d$ and all its inner monomials have minus sign, nor of an reciprocal octanomial of degree $D_2$ such that $D_2<1.5d$ and all its monomials have plus sign.
\end{conjecture}
\paragraph*{Acknowledgment.}Partially supported by Serbian Ministry of Education and Science, Project 174032.






\bibliography{Stankovbibfile}

\end{document}